%% file: p-rotor-walk-arxiv-april.tex
\documentclass[letterpaper,11pt]{article}
\usepackage[utf8x]{inputenc}

\usepackage[noTeX]{mmap}
\usepackage[left=1.35in, right=1.35in, bottom = 1.4in, top = 1.4in]{geometry}
\usepackage{hyperref} 
\hypersetup{colorlinks}
\usepackage{amsmath,amsfonts, amsthm, amssymb}
\usepackage{color}
\usepackage{enumerate}
\usepackage{hypbmsec}
\usepackage{bbm}
\usepackage{dsfont}
\usepackage{mathtools}
\usepackage{graphicx}
\usepackage{caption}
\usepackage[section]{placeins}
\usepackage{tikz}
\usetikzlibrary{calc}
\usetikzlibrary{decorations.pathreplacing}
\usetikzlibrary{arrows}
\usepackage{pgffor}
\usepackage{longtable}
\usepackage{comment}

\usepackage{fancyhdr}
\theoremstyle{plain}
\newtheorem{theorem}{Theorem}[section]
\newtheorem{proposition}[theorem]{Proposition}
\newtheorem{corollary}[theorem]{Corollary}
\newtheorem{lemma}[theorem]{Lemma}
\newtheorem{definition}[theorem]{Definition}

\theoremstyle{definition}
\newtheorem{remark}[theorem]{Remark}

\newtheorem{notation}[theorem]{Notation}

\DeclarePairedDelimiter\floor{\lfloor}{\rfloor}
\DeclarePairedDelimiter\abs{\lvert}{\rvert}%

\newcommand{\B}{\mathcal{B}}
\newcommand{\N}{\mathbb{N}}
\newcommand{\Z}{\mathbb{Z}}
\newcommand{\X}{\mathsf{X}}
\newcommand{\Y}{\mathsf{Y}}

\newcommand{\M}{\mathsf{M}}
\newcommand{\m}{\mathsf{m}}
\newcommand{\f}{\mathsf{W}}
\newcommand{\Zn}{\mathsf{Z}}
\newcommand{\E}{\mathbb{E}}
\newcommand{\F}{\mathcal{F}}
\newcommand{\Pb}{\mathbb{P}}

\newcommand{\indicator}{\mathds{1}}

\newcommand{\dequal}{\overset{\mathcal{D}}{=}}

\newcommand{\weaklyto}{\xrightarrow{\mathcal{D}}}

\newcommand{\pref}[1]{\hyperref[#1]{[p. \pageref{#1}]}}

\newcommand{\old}[1]{}

\title{Interpolating between random walk and rotor walk}

\author{
Wilfried Huss\footnote{Graz University of Technology,
\texttt{huss@math.tugraz.at}. Research supported by the Austrian Science Fund (FWF): 
\href{http://pf.fwf.ac.at/en/research-in-practice/in-the-spotlight-schroedinger/list-of-schroedinger-fellows/2014/8523}{J3628-N26}},
Lionel Levine\footnote{Cornell University, \texttt{levine@math.cornell.edu}. Research supported by NSF grant \href{http://www.nsf.gov/awardsearch/showAward?AWD_ID=1455272}{DMS-1455272} and a Sloan Fellowship.},
Ecaterina Sava-Huss\footnote{Graz University of Technology, \texttt{sava-huss@tugraz.at}. Research supported by the Austrian Science Fund (FWF):
\href{http://pf.fwf.ac.at/en/research-in-practice/in-the-spotlight-schroedinger/list-of-schroedinger-fellows/2014/210946}{J3575-N26}}
}

\date{April 7, 2016}

\begin{document}

\maketitle

\begin{abstract}
We introduce a family of stochastic processes on the integers, depending on a parameter $p \in [0,1]$ and interpolating between the deterministic rotor walk ($p=0$) and the simple random walk ($p=1/2$). 
This \textbf{p-rotor walk} is not a Markov chain but it has a \emph{local} Markov property: for each $x \in \Z$ the sequence of successive exits from $x$ is a Markov chain.
The main result of this paper identifies the scaling limit of the p-rotor walk 
with two-sided i.i.d.\ initial rotors. 
The limiting process takes the form $\sqrt{\frac{1-p}{p}} X(t)$, where $X$
is a doubly perturbed Brownian motion, that is, it  satisfies the implicit equation
\begin{equation}
\label{eq:implicit}
X(t) = \B(t) + a \sup_{s\leq t} X(s) + b \inf_{s\leq t} X(s)
\end{equation}
for all $t \in [0,\infty)$. Here $\B(t)$ is a standard Brownian motion and $a,b<1$ are constants depending on the marginals of the initial rotors on $\N$ and $-\N$ respectively.  
Chaumont and Doney have shown that equation \eqref{eq:implicit} has a pathwise unique solution $X(t)$, and that the solution is almost surely continuous and adapted to the natural filtration of the Brownian motion \cite{chaumont_doney_1999}.  Moreover, $\limsup X(t) = +\infty$ and $\liminf X(t) = -\infty$ \cite{chaumont_doney_2000}.  This last result, together with the main result of this paper, implies that the p-rotor walk is recurrent for any two-sided i.i.d.\ initial rotors and any $0<p<1$.
\end{abstract}



\section{Introduction}\label{sec:intro}

In a \emph{rotor walk} on a graph, the exits from each vertex follow a prescribed periodic sequence.
In the last decade Propp \cite{propp_lecture}, Cooper and Spencer \cite{CS}, and Holroyd and Propp \cite{holroyd_propp} 
developed close connections 
between the behavior of rotor walk and the first-order properties of random walk.  On finite graphs and on $\Z$, rotor walk approximates the $n$-step distribution, 
stationary distribution, expected hitting times and harmonic measure of random walk to within a bounded additive error.  On other infinite graphs, especially in questions concerning recurrence and transience, rotor walk can have different behavior from random walk \cite{landau_levine, angel_holroyd, angel_holroyd_rec_config, huss_sava_trans_directed_covers, florescu_ganguly_levine_peres_escapes,  huss_muller_sava_gw, friedrich_katzmann_krohmer}.

\old{An argument of Schramm shows 
that rotor walk is ``at least as recurrent'' as random walk \cite{florescu_ganguly_levine_peres_escapes}, but in many instances rotor walk can be 
significantly ``more recurrent'' \cite{angel_holroyd, angel_holroyd_rec_config, huss_sava_trans_directed_covers, huss_muller_sava_gw}.

For example, it is believed that the uniform rotor walk on $\Z^2$ (i.e., the rotor walk with independent and uniform 
initial rotors) visits only $O(n^{2/3})$ distinct sites in $n$ steps
\cite{PDDK,KD,florescu_levine_peres}.
In spite of strong numerical evidence for this conjecture, it remains open even to prove that this number is $o(n)$, or that the uniform rotor walk on $\Z^2$ is recurrent!
The difficulties in approaching this conjecture are of the same type as those encountered in the Lorenz mirror model \cite[\textsection13.3]{grimmett}: all available randomness is present in the initial condition.  }

An interesting question is how to define a modification of rotor walk that approximates well not just the mean, but also the second and higher moments of some observables of random walk. Propp (personal communication) has proposed an approach involving multiple species of walkers. In the current work we explore a rather different approach to this question. We interpolate between rotor and random walk by introducing a parameter $p \in [0,1]$. During one step of the \emph{p-rotor walk}, if the current rotor configuration is $\rho : \Z \to \{-1,+1\}$ and the current location of the walker is $x \in \Z$, then we change the sign of $\rho(x)$ with probability $1-p$, and then move the walker one step in the direction of $\rho(x)$. 

More formally, we define a Markov chain on pairs $(\X_n,\rho_n) \in \Z \times \{-1,+1\}^\Z $ by setting 
\begin{equation}\label{eq:rho_n}
 \rho_{n+1}(x)=
\begin{cases}
 \rho_{n+1}(x)  & \text{for } x\neq \X_n,\\
 B_n\rho_n(\X_n) & \text{for } x= \X_n
\end{cases}
\end{equation}
where $B_0, B_1, \ldots$ are independent with $P(B_n = 1) = p = 1- P(B_n = -1)$, for all $n\in\mathbb{N}$. Then we set
\begin{equation}\label{eq:xn_incr}
 \X_{n+1} = \X_n + \rho_{n+1}(\X_n).
\end{equation}
Here $\rho_n$ represents the rotor configuration and $\X_n$ the location of the 
walker after $n$ steps. The parameter $p$ has the following interpretation: at each time step the rotor at the walker's current location is \textit{broken} and 
fails to flip with probability $p$, independently of the past. Note that if the walker visited $x$ at some previous time, then the rotor $\rho_n(x)$ indicates the direction of the most recent exit from $x$, but it retains no memory of whether it was broken previously.

The p-rotor walk is an example of a \emph{stochastic Abelian network} as proposed in \cite{bond_levine_an1},
moreover it is also a special case of an \emph{excited random walk with Markovian cookie stacks} \cite{kosygina_peterson_2015}. The model studied there does not include p-rotor walk as
a special case due to the ellipticity assumption made in this paper.
The pair $(\X_n,\rho_n)$ is a Markov chain, but $(\X_n)$ itself is not a Markov chain unless $p \in \{1/2, 1\}$. 
If $p=1/2$ then 
$(\X_n)$ is a simple random walk on $\Z$. If $p=1$ then $(\X_n)$ deterministically follows the initial rotors $\rho_0$. If $p=0$ then $(\X_n)$ is a rotor walk in the usual sense.
The aim of the current work is to prove that the p-rotor walk on $\Z$ with two-sided i.i.d.\ configuration, when 
properly rescaled, converges weakly to a doubly-perturbed Brownian motion.

\subsection{Main results}\label{subsec:main_results}

We prove a scaling limit theorem for p-rotor walks $(\X_n)$ with random initial rotor configuration on
$\Z$ as following.  The two-sided initial condition we will consider depends on  parameters $\alpha,\beta \in [0,1]$: the initial rotors $\big(\rho_0(x)\big)_{x\in \Z}$ are independent with
\begin{equation}\label{eq:random_cfg}
 \rho_0(x)=
\begin{cases}
-1 & \text{ with probability } \beta, \text{ if } x < 0\\
 1 & \text{ with probability } 1-\beta, \text{ if } x < 0\\
 -1 & \text{ with probability } 1/2, \text{ if } x = 0\\
 1 & \text{ with probability } 1/2, \text{ if } x = 0\\
 1 & \text{ with probability } \alpha, \text{ if } x > 0\\
 -1 & \text{ with probability } 1-\alpha, \text{ if } x > 0.
\end{cases}
\end{equation}
That is, initially, all rotors on the positive integers point to the right
with probability $\alpha$ and to the left with probability $1-\alpha$. Similarly, on the negative integers, initially all rotors point to the left with probability $\beta$ and to the right with probability $1-\beta$.
We can change any finite number of rotors in the initial configuration \eqref{eq:random_cfg}, and the scaling limit of the p-rotor walk will still be the same. See Remark \ref{rem:in_cfg} for more details.
For every $\alpha,\beta\in [0,1]$, for the configuration \eqref{eq:random_cfg} we shall use the name \textit{$(\alpha,\beta)$-random initial
configuration}. 

For a continuous time process $X(t)$ we denote by 
\begin{equation*}
X^\mathsf{sup}(t) = \sup_{s\leq t} X(s)\quad \text{and by} \quad X^\mathsf{inf}(t) = \inf_{s\leq t} X(s)
\end{equation*}
the running supremum and the infimum of $X(t)$ respectively. Denote by $\big(\B(t)\big)_{t\geq 0}$ the standard Brownian motion started at $0$.

\begin{definition}
\label{def:perturbed_brownian_motion}
A process $\mathcal{X}_{a,b}(t)$ is called an \emph{$(a,b)$-perturbed Brownian motion}
with parameters
$a,b\in\mathbb{R}$, if $\mathcal{X}_{a,b}(t)$ is a solution of the implicit equation
\begin{equation}
\label{eq:perturbed_brownian_motion}
\mathcal{X}_{a,b}(t) = \B(t) + a \mathcal{X}_{a,b}^\mathsf{sup}(t) + b\mathcal{X}_{a,b}^\mathsf{inf}(t)
\end{equation}
for all $t\geq 0$.
\end{definition}
The process $\mathcal{X}_{a,b}(t)$ has been called a \emph{doubly perturbed Brownian motion} \cite{davis:annals, carmona_petit_yor_1998}.
For $a,b\in (-\infty,1)$ equation
\eqref{eq:perturbed_brownian_motion} has a pathwise
unique solution; moreover, the solution is almost surely continuous and is adapted to the natural filtration of the Brownian motion $\B(t)$ \cite[Theorem 2]{chaumont_doney_1999}. 
for additional results in this direction
see also
For other important properties of the doubly perturbed Brownian motion we refer to \cite{chaumont_doney_2000}.
We are now ready to state our main result.

\begin{theorem}\label{thm:scaling_lim}
For all $p\in (0,1)$ and all $\alpha,\beta\in [0,1]$, the p-rotor walk $(\X_n)$ on $\Z$
with $(\alpha,\beta)$-random initial configuration as in \eqref{eq:random_cfg}, after rescaling converges 
weakly to an $(a,b)$-perturbed Brownian motion
\begin{equation*}
\bigg\{ \frac{\X(nt)}{\sqrt{n}},\ t\geq 0 \bigg\}\weaklyto
\bigg\{\sqrt{\frac{1-p}{p}} \mathcal{X}_{a,b}(t),\ t\geq 0 \bigg\} \quad \text{ as } n\to\infty,
\end{equation*}
with
$$a = \frac{\alpha(2p-1)}{p} \quad \text{ and }\quad  b = \frac{\beta(2p-1)}{p}.$$
\end{theorem}
Note that
\begin{equation*}
1-a = \frac{\alpha(1-p) + p (1-\alpha)}{p}>0\quad\text{and}\quad 1-b = \frac{\beta(1-p) + p (1-\beta)}{p} > 0,
\end{equation*}
hence $a,b < 1$ for all $p\in (0,1)$ and all $\alpha,\beta\in [0,1]$, which ensures the existence and uniqueness of the solution of the equation \eqref{eq:perturbed_brownian_motion}. Moreover $a$ and $b$ have the same sign: $a,b \geq 0$ if $p\geq 1/2$ and $a,b < 0$ if $p < 1/2$.

Doubly perturbed Brownian motion arises as a weak limit of several other discrete processes:
perturbed random walks \cite{davis:annals}; \textit{pq walks} \cite{davis_bm_rw}; \textit{asymptotically free walks} \cite{toth:rayknight}; and certain \textit{excited walks} \cite{dolgopyat_kosygina_2012}.  It is also a degenerate case of the ``true self-repelling motion'' of T{\'o}th and Werner \cite{toth_werner}.

If we take $\beta=0$ in \eqref{eq:random_cfg}, then all rotors on
the negative integers point initially towards the origin.  In this special case the perturbed Brownian motion $\mathcal{X}_{a,b}$ with $b=0$ has a well-known explicit formula: it is a linear combination of a standard brownian motion $\B(t)$ and its running maximum $\mathcal{M}(t) = \sup_{s \leq t} \B(s)$.

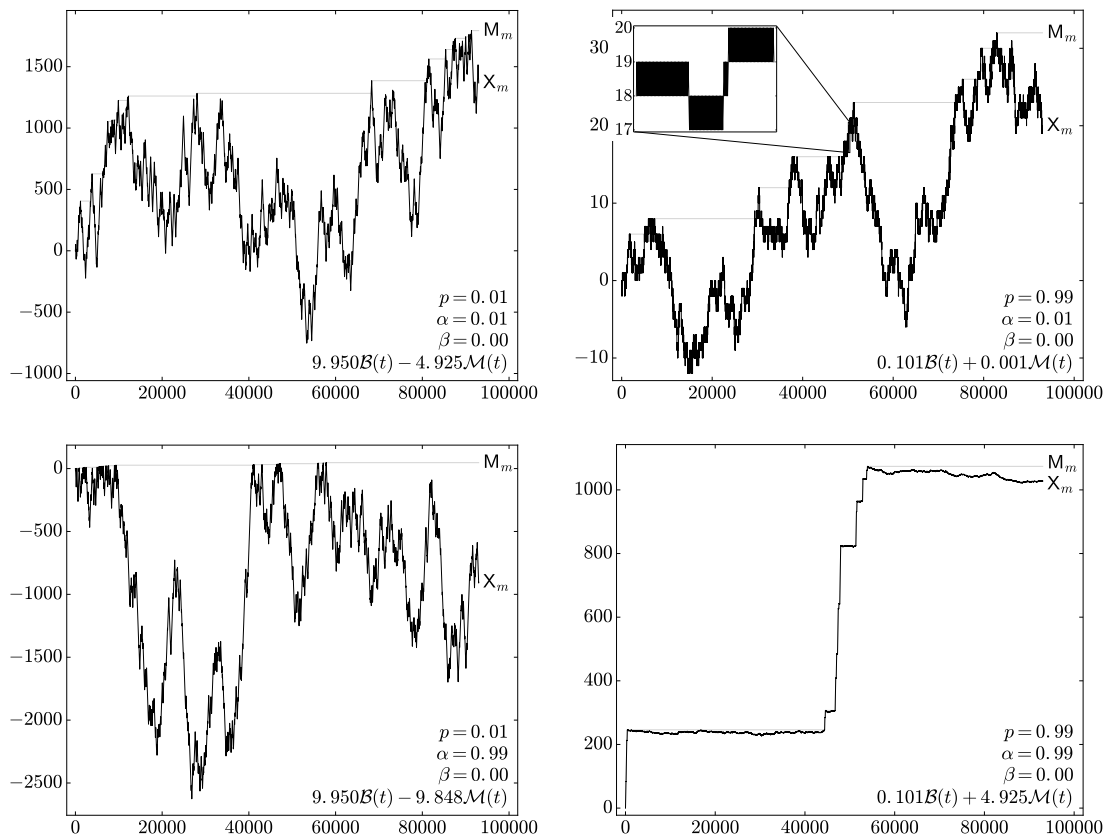
\begin{figure}[t]
\input{sim_trajectory_arxiv.tex}
\caption{\label{fig:path_behavior} Sample paths of the p-rotor walk $\X_n$ for various extreme cases of the parameters $p$ and $\alpha$. In each case $\beta=0$, so the scaling limit of $(\X_n)$ is a linear combination of a Brownian motion $\B(t)$ and its running maximum $\mathcal{M}(t)$.}
\end{figure}

\begin{corollary}\label{cor:one-sided}
For all $p\in (0,1)$ and all $\alpha\in [0,1] $, the rescaled p-rotor walk $(\X_n)$ 
with $(\alpha,0)$-random initial configuration, with $\beta=0$ in \eqref{eq:random_cfg}, converges weakly to a one-sided perturbed Brownian motion
\begin{equation*}
\bigg\{ \frac{\X(nt)}{\sqrt{n}},\ t\geq 0\bigg\}\xrightarrow{\mathcal{D}}
\bigg\{\sqrt{\frac{1-p}{p}}\big(\mathcal{B}(t)+\lambda \mathcal{M}(t)\big),\ t\geq 0 \bigg\}
\quad \text{ as } n\to\infty,
\end{equation*}
where \[ \lambda = \lambda_{p,\alpha}=\frac{\alpha(2p-1)}{\alpha(1-p)+p(1-\alpha)}. \]
\end{corollary}
The process arising as the scaling limit in this result has the following intuitive 
interpretation: it behaves as a Brownian motion except when it is at its maximum,
when it gets a push up if $\lambda>0$ or a push down if $\lambda<0$. 
By symmetry, we get the same scaling limit in the case $\alpha=0$, with the minimum of the Brownian motion replacing the maximum in Corollary \ref{cor:one-sided}.

The scaling limit of the p-rotor walk (Theorem~\ref{thm:scaling_lim}) along with the fact (proved in \cite{chaumont_doney_2000})
that the doubly perturbed Brownian motion $\mathcal{X}_{a,b}$ satisfies $\limsup_{t\to\infty} \mathcal{X}_{a,b}(t) = +\infty$ and $\liminf_{t\to\infty} \mathcal{X}_{a,b}(t) = -\infty$ almost surely, implies the following.

\begin{corollary}
\label{prop:recurrence}
For all $p\in (0,1)$ and all $\alpha,\beta\in [0,1]$, the  p-rotor walk $(\X_n)$ with
$(\alpha,\beta)$-random initial configuration \eqref{eq:random_cfg} is recurrent on $\Z$.
\end{corollary}

Figure \ref{fig:path_behavior} shows sample paths of the p-rotor walk $(\X_n)$ in the case $\beta = 0$
and various extreme cases of the parameters $p$ and $\alpha$. The parameter values and the formula for the corresponding scaling limit appear in the corner of each picture.
In the pictures on the left ($p=0.01$) the p-rotor walk takes long sequences of steps in the same direction because the rotors are rarely broken. On the
right side ($p=0.99$) the rotors are broken most of the time and the walk spends most of its time trapped in a cycle alternating between two neighboring sites, as seen in the inset in the picture on the upper right.
If $\alpha$ is close to $1$, so that most rotors initially point to the right, then the maximum increases slowly 
if $p$ is small (bottom left).
On the other hand, for $p$ and $\alpha$ both close to $1$ (bottom right), 
when the process forms a new maximum it tends to take many consecutive steps to the right.

In the course of the evolution of the process $(\X_n)$, the rotor configuration $\rho_n$ has always a simple form. Let
\begin{equation}
 \M_n=\max_{k\leq n} \, \X_k \quad \text{and}  \quad \m_n=\min_{k\leq n} \, \X_k
\end{equation}
be the running maximum and running minimum of $(\X_k)$ up to time $n$ respectively.
For all $\mathsf{m}_n \leq x \leq \mathsf{M}_n$, if $x \neq \X_n$ then the rotor $\rho_n(x)$ necessarily points from $x$ in the direction of $\X_n$. Indeed, if $x$ was visited before time $n$ then $\rho_n(x)$ points in the direction of the most recent exit from $x$.
On the other hand, for all $x \not\in \{\mathsf{m}_n,\ldots, \mathsf{M}_n\}$, the rotors remain in their random initial state $\rho_0$, see Figure \ref{fig:environment}.
Hence whenever the process visits a vertex $x$ for the first time, there will be some perturbation if the rotor at
$x$ initially does not point toward the origin.

\begin{figure}[h]
\centering
\input{figure_environment.tex}
\caption{\label{fig:environment} Each rotor $\rho_n(x)$ is shown by an arrow pointing left or right from $x$, accordingly as $\rho_n(x)$ is $-1$ or $+1$.  The rotors in the visited interval $[\mathsf{m}_n, \mathsf{M}_n]$ always point towards the current position $\X_n$ of the walker.}
\end{figure}
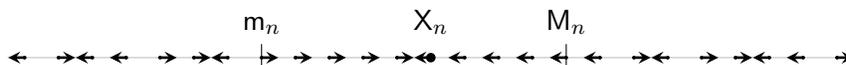

\begin{notation}
\label{not:lin_inter}
Discrete time processes will be denoted $(\X_n), (\Y_n)$, etc., omitting the subscript $n\geq 0$. Square brackets $[ \cdot ]$ denote an event and $\indicator [ \cdot ]$ its indicator.
For all probabilities related to p-rotor walks, we omit the starting point $0$, writing just $\Pb$ instead of $\Pb_0$.
For a discrete time process $(\X_n)$, we denote by $\X(t)$ its linear interpolation to real times $t \in [0,\infty)$
\begin{equation*}
\X(t) = \X_{\floor{t}} + (t - \floor{t})(\X_{\floor{t}+1} - \X_{\floor{t}}).
\end{equation*}
For the scaling limit we look
at the sequence of random continuous functions $\X(kt)/\sqrt{k}$ on the interval $[0,\infty)$. 
Let $\mathcal{C}[0,\infty )$ and $\mathcal{C}[0,T]$ (for $0<T<\infty$) be the spaces of continuous functions $[0,\infty ) \to \mathbb{R}$ and $[0,T] \to \mathbb{R}$, respectively.
We write $\weaklyto$ 
for weak convergence on $\mathcal{C}[0,T]$ with respect to the norm $\| f \| = \sup_{0 \leq t \leq T} |f(t)|$. 
We say that a sequence of random functions $X_k \in \mathcal{C}[0,\infty)$ converges weakly to $X\in\mathcal{C}[0,\infty )$
if the restrictions converge weakly: $X_k|_{[0,T]} \weaklyto X|_{[0,T]}$ in $\mathcal{C}[0,T]$ for all $T > 0$; see \cite[Page 339]{durrett_book}. 
\end{notation}

The rest of the paper is structured as follows.
In Section \ref{sec:sc-lim} we prove the main theorem, which is based on the decomposition of the p-rotor walk path
into a martingale term and a compensator.  The compensator decomposes as a linear combination of three pieces: $\m_n$, $\M_n$, and $\X_n$ itself. 

We apply a version of the functional central limit theorem to show that the martingale term converges weakly to a Brownian motion with a constant factor \emph{different} from $\sqrt{\frac{1-p}{p}}$. The true constant factor $\sqrt{\frac{1-p}{p}}$ appears after we correct for the $\X_n$ term in the compensator.

The proof of the scaling limit for $(\X_n)$ requires the understanding of the scaling limit (and
recurrence) of the \emph{native} case, which is the p-rotor walk with $\alpha=\beta=0$ in the initial
configuration \eqref{eq:random_cfg}. This will be done in Subsection \ref{sec:cor_rw}.
We conclude with several questions and possible extensions of our model in Section \ref{sec:questions}.

\section{Scaling limit}
\label{sec:sc-lim}

We decompose first the p-rotor walk into a martingale and a compensator, and we prove that
$(\X_n)$ does not grow too fast, i.e.~it is tight. 
A similar approach has been used in \cite{dolgopyat_kosygina_2012} to deduce the scaling limit of a recurrent particular
case of an excited random walk on $\Z$. 

Let $\Delta_k = \X_{k+1}-\X_{k}$ for $k\geq 0$ and denote by $\F_k = \sigma(\X_0,\ldots,\X_k)$
the natural filtration of the p-rotor walk $(\X_n)$. Then, for all $n\geq 1$ we can write 
\begin{equation}
\label{eq:martingale_decomposition}
\X_n = \sum_{k=0}^{n-1} \Delta_k = \Y_n + \Zn_n,
\end{equation}
with 
\begin{equation}
\label{eq:Yn_Zn}
\Y_n = \sum_{k=0}^{n-1}\big( \Delta_k - \E[\Delta_k| \F_{k}] \big)
\quad\text{and}\quad \Zn_n = \sum_{k=0}^{n-1} \E[\Delta_k| \F_{k}].
\end{equation}
Let $$\xi_k = \Delta_k - \E[\Delta_k| \F_{k}].$$
Since $\xi_k \in \F_{k+1}$ and
$\E[\xi_k| \F_{k}] = 0$ for all $k\geq 0$, the sequence 
$\big\{\xi_k$, $\F_{k+1}\big\}_{k\geq 0}$ is 
a martingale difference sequence. 
Therefore the process $(\Y_n)$ is a martingale with respect to the filtration $\F_n$.
We will use the following functional limit theorem for martingales, see \textsc{Durrett}
\cite[Theorem 7.4]{durrett_book}.

\begin{theorem}[Martingale central limit theorem]
\label{thm:martingale_clt}
Suppose $\{\xi_k, \F_{k+1}\}_{k\geq 1}$ is a martingale difference sequence and let
$\Y_n = \sum_{1\leq k \leq n} \xi_k$ and $\mathsf{V}_n = \sum_{1\leq k \leq n} \E\big[\xi_k^2|\F_{k}\big]$. If
\begin{enumerate}[(a)]
\item \label{martingal_clt_a} $\displaystyle\frac{1}{n} \sum_{1\leq k\leq n} \E\big[\xi_k^2 \indicator\{\abs{\xi_k} > \epsilon \sqrt{n}\}\big] \to 0$ for all $\epsilon > 0$, as $n\to\infty$
\item \label{martingal_clt_b} $\displaystyle\frac{\mathsf{V}_n}{n} \to \sigma^2 > 0$ in probability, as $n\to\infty$ and
\end{enumerate}
then $\frac{\Y(nt)}{\sqrt{n}}$ converges weakly to a Brownian motion:
\begin{equation*}
\left\{\frac{\Y(nt)}{\sqrt{n}},\, t\in [0,1] \right\} \weaklyto 
\big\{\sigma \B(t),\, t\in [0,1] \big\},\quad \text {as } n \to \infty.
\end{equation*}
\end{theorem}

In order to prove the scaling limit theorem for $(\X_n)$, 
we first look at the compensator $\mathsf{Z}_n$ in the decomposition \eqref{eq:martingale_decomposition} of $\X_n$.

\begin{proposition}
\label{prop:z_n}
The compensator $\mathsf{Z}_n$ in the decomposition \eqref{eq:martingale_decomposition} of the p-rotor walk is equal to
\begin{equation*}
 \mathsf{Z}_n=(2p-1)(2\beta \m_{n-1}+2\alpha \M_{n-1} - \X_{n-1}),\quad \text{for all } n\geq 1.
\end{equation*}
\end{proposition}

\begin{proof}
Recall, from \eqref{eq:xn_incr} that
$\Delta_k=\X_{k+1}-\X_{k}=B_{k}\rho_{k}(\X_{k})$ and $\rho_{k}(\X_{k})\in\F_{k}$ if
$\X_{k}$ has been already visited. If $\X_{k}$ has not been visited before time $k\geq 1$, that is,
if $\X_{k} < \m_{k-1}$ or $\X_k > \M_{k-1}$, then $\rho_k(\X_{k}) = \rho_0(\X_{k})$ and thus
is independent of $\F_k$. The working state of the rotor $B_k$ is independent of $\F_{k}$, and
we have $\mathbb{E}[B_{k}]=2p-1$. It follows that for $k\geq 1$
\begin{align*}
\mathbb{E}[\Delta_k|\F_{k}] & =(2p-1)(1-2\beta)\indicator\{\X_{k}<\m_{k-1}\}\\
 & + \mathbb{E}[\Delta_k|\F_k]\indicator\{\m_{k-1}\leq \X_k \leq \M_{k-1}\}\\
 & + (2p-1)(2\alpha-1)\indicator\{\X_k>\M_{k-1}\}.
\end{align*}
Using the fact that
 $\m_n = -\sum_{k=1}^n \indicator\{\X_{k}<\m_{k-1}\}$ and 
 $\M_{n}=\sum_{k=1}^n \indicator\{\X_{k}>\M_{k-1}\} $,
gives
\begin{align*}
\mathsf{Z}_n=\E[\Delta_0] + \sum_{k=1}^{n-1}\mathbb{E}[\Delta_k|\F_{k}] & =(2p-1)\big\{(2\beta-1)\m_{n-1}+(2\alpha-1)\M_{n-1}\big\} \\ 
& + \sum_{k=1}^{n-1}\mathbb{E}[\Delta_k|\F_k]\indicator\{\m_{k-1}\leq \X_k \leq \M_{k-1}\}.
\end{align*}
On the other hand, on the event $\{\m_{k-1}\leq \X_k \leq \M_{k-1}\}$
\begin{equation*}
 \Delta_{k}=\Delta_{k-1}\indicator\{B_k=-1\}-\Delta_{k-1}\indicator\{B_k=1\},
\end{equation*}
with $\Delta_{k-1} \in \F_k$. It follows that
\begin{align*}
\sum_{k=1}^{n-1} \mathbb{E}[\Delta_k|\F_k]\indicator\{\m_{k-1}\leq \X_k \leq \M_{k-1}\}
=(1-2p) \sum_{k=1}^{n-1}\Delta_{k-1}\indicator\{\m_{k-1}\leq \X_k \leq \M_{k-1}\} .
\end{align*}
Let us denote by $\mathsf{C}_n$ the quantity
\begin{equation*}
\mathsf{C}_n=\m_{n-1}+\sum_{k=1}^{n-1}\Delta_{k-1}\indicator\{\m_{k-1}\leq \X_k \leq \M_{k-1}\} +\M_{n-1}.
\end{equation*}
The compensator $\mathsf{Z}_n$ can then be rewritten as
\begin{equation*}
\mathsf{Z}_n=(2p-1)(2\beta \m_{n-1}-\mathsf{C}_n+2\alpha \M_{n-1} ).
\end{equation*}
It remains to show that $\mathsf{C}_n=\X_{n-1}$, for all $n\geq 1$. This is a straightforward calculation.
We have
\begin{align*}
\X_{n-1}-\mathsf{C}_n& =\sum_{k=0}^{n-2}\Delta_k-\mathsf{C}_n=\sum_{k=1}^{n-1}\Delta_{k-1}-\mathsf{C}_n\\
& =\sum_{k=1}^{n-1}\Delta_{k-1}\indicator\{\m_{k-1}\leq \X_{k}\leq \M_{k-1}\}
+\sum_{k=1}^{n-1}\Delta_{k-1}\indicator\{\X_{k}<\m_{k-1}\}\\
& +\sum_{k=1}^{n-1}\Delta_{k-1}\indicator\{\X_k>\M_{k-1}\}- \mathsf{C}_n.
\end{align*}
On the event $\{\X_{k}<\m_{k-1}\}$, we have that $\X_{k-1}=\m_{k-1}$ and $\Delta_{k-1}=-1$.
On the event $\{\X_{k}>\M_{k-1}\}$, we have that $\X_{k-1}=\M_{k-1}$ and $\Delta_{k-1}=1$.
Therefore
\begin{align*}
 \X_{n-1}-\mathsf{C}_n=-\sum_{k=1}^{n-1}\indicator\{\X_{k}<\m_{k-1}\}+
 \sum_{k=1}^{n-1}\indicator\{\X_k>\M_{k-1}\}-\m_{n-1}-\M_{n-1}=0,
\end{align*}
which completes the proof.
\end{proof}
\begin{proposition}
\label{prop:xn_decomposition}
 For all $p\in(0,1)$ and all $\alpha,\beta\in[0,1]$, the p-rotor walk
 $(\X_n)$ with $(\alpha,\beta)$-random initial configuration $\rho_0$ as in  \eqref{eq:random_cfg}
 satisfies:
 \begin{equation}
  \label{eq:xn_decomposition}
  \X_n = \f_n + a \M_n + b \m_n, \text{ for all } n\geq 1,
 \end{equation}
with
\begin{equation}
 a = \frac{\alpha(2p-1)}{p} \quad \text{and} \quad b=\frac{\beta(2p-1)}{p}
\end{equation}
and $\f_n$ given by
  \begin{equation}
  \label{eq:f_n}
  \f_n = \frac{1}{2p}\big(\Y_{n+1} - \Delta_n\big).
  \end{equation}
\end{proposition}
\begin{proof}
From \eqref{eq:martingale_decomposition} and Proposition \ref{prop:z_n} we get
\begin{equation*}
\X_n=\Y_n+ (2p-1)(2\beta \m_{n-1}+2\alpha \M_{n-1} -\X_{n-1}),
\end{equation*}
which together with $\X_n=\X_{n-1}+\Delta_{n-1}$ gives the following representation
of the p-rotor walk in terms of its minimum and maximum
\begin{equation*}
 2p\X_{n-1}=\Y_n-\Delta_{n-1}+2\alpha(2p-1)\M_{n-1}+2\beta(2p-1)\m_{n-1}.
\end{equation*}
Dividing by $2p$ and reindexing gives the claim.
\end{proof}

We focus next our attention on the martingale term $\Y_n$ in the decomposition \eqref{eq:martingale_decomposition} of the
p-rotor walk. First we consider briefly the special case of $\alpha = \beta = 0$, which is particularly simple to understand
and whose properties will be used in the behavior of the general case.

\subsection{Native environment}
\label{sec:cor_rw}
In the native case $\alpha = \beta = 0$ our initial rotor configuration has the form
\begin{equation}\label{eq:yn_rot_cfg}
 \rho_0(x) = \begin{cases}
            -1 & \text{ for } x > 0, \\
            -1 \quad\text{with probability 1/2} & \text{ for } x = 0,\\
            1 \quad\text{with probability 1/2} & \text{ for } x = 0,\\
             1 & \text{ for } x < 0,
             \end{cases}
\end{equation}
Denote by $(\mathsf{U_n})$ the p-rotor walk started with this initial configuration.
We shall call $(\mathsf{U_n})$ the \emph{native p-rotor walk}.
As mentioned in the introduction, in the previously visited region $\{\mathsf{m}_n,\ldots,\mathsf{M}_n\}$,
the configuration $\rho_n$ of the p-rotor walk $(\X_n)$ points in the direction of the current position. Therefore,
in the visited region the p-rotor walk $(\X_n)$ behaves exactly like $(\mathsf{U}_n)$. The process $(\mathsf{U}_n)$
is an easy special case of a \emph{correlated random walk} which has been studied in greater generality in \cite{enriquez:corr_rw}.

\begin{definition}
\label{d.correlatedRW}
A \emph{correlated random walk} on $\Z$ with persistence $q\in(0,1)$ is a nearest
neighbour random walk, such that with probability $q$ the direction of a step is the same as
the direction of the previous step. If $q=1/2$, then it is a simple random walk. 
\end{definition}

Since all rotors always point to the current position $\mathsf{U}_n$ of the walker,
the rotor $\rho_n(\mathsf{U_n})$ points towards the previous position $\mathsf{U}_{n-1}$.
Thus the direction of movement changes only if the rotor at time $n$ is broken (i.e. $B_n=1$), which
happens with probability $p$. Thus $(\mathsf{U}_n)$ is a correlated random walk with
persistence $1-p$. 

It is easy to see that $(\mathsf{U}_n)$ when properly rescaled converges weakly to a
Brownian motion. We give a quick proof of this fact for completeness.
\begin{proposition}
\label{prop:undisturbed_clt}
For every $p\in(0,1)$, the native p-rotor walk $(\mathsf{U}_n)$ with initial configuration as in \eqref{eq:yn_rot_cfg}
when rescaled by $\sqrt{n}$, converges weakly on $\mathcal{C}[0,1]$ to a Brownian motion:
\begin{equation*}
\left\{\frac{\mathsf{U}(nt)}{\sqrt{n}}, t\in [0,1]\right\} \weaklyto \left\{\sqrt{\frac{1-p}{p}} 
\B(t),\, t\in [0,1]\right\},\qquad\text{as } n\to\infty.
\end{equation*}
\end{proposition}
\begin{proof}
From Proposition \ref{prop:xn_decomposition} with $\alpha=\beta=0$ we get that for all $n\geq 1$
\begin{equation*}
\mathsf{U}_n = \frac{1}{2p}\big(\Y_{n+1} - \Delta_n\big).
\end{equation*}
Since $\Delta_n\in\{-1,+1\}$, the process $(2p\mathsf{U}_n)$ has the same scaling limit as $\Y_n = \sum_{k=0}^{n-1}\xi_k$,
with $\xi_k = \Delta_k - \E[\Delta_k|\F_k]$. Because $\Y_n$ is a martingale, we can apply Theorem \ref{thm:martingale_clt}.
The first condition of Theorem \ref{thm:martingale_clt} is satisfied since $\xi_k$ is uniformly bounded for all $k\geq 0$.
Thus, we only have to show convergence 
of the quadratic variation process $(\mathsf{V}_n)$. Using  the fact that $\mathsf{U}_n$ is a nearest 
neighbour walk, the following holds
\begin{equation}\label{eq:vn_variation}
\begin{aligned}
\mathsf{V}_n  = & \sum_{k=1}^{n}\mathbb{E}[\xi_k^2|\F_{k}]=\sum_{k=1}^{n}
\mathbb{E}\big[(\Delta_k-\mathbb{E}[\Delta_k|\F_{k}])^2|\F_{k}\big] \\
=& \sum_{k=1}^{n}\mathbb{E}[(\Delta_k^2-2\Delta_k\mathbb{E}[\Delta_k|\F_{k}]+\mathbb{E}[\Delta_k|\F_{k}]^2)|\F_{k}\mathbb]\\
= & n-\sum_{k=1}^{n}\mathbb{E}[\Delta_k|\F_{k}]^2,
\end{aligned}
\end{equation}
On the other hand, from
equation \eqref{eq:xn_incr} we have the equality
$\Delta_k=\X_{k+1}-\X_{k}=B_{k}\rho_{k}(\X_{k})$ where $\rho_{k}(\X_{k})\in\F_{k}$
and $B_{k}$ independent of $\F_{k}$ with $\mathbb{E}[B_{k}]=(2p-1)$. Hence
\begin{equation*}
\mathbb{E}[\Delta_k|\F_{k}]^2 = (2p-1)^2\rho_k(\X_k)^2 = (2p-1)^2.
\end{equation*}
Then $$\frac{\mathsf{V}_n}{n} = 4p(1-p),$$ from which the claim immediately follows.
\end{proof}

\subsection{General environment}

We now treat the general case of an $(\alpha,\beta)$-random initial configuration with $\alpha,\beta \in [0,1]$.
In order to check that $(\Y_n)$ as defined in \eqref{eq:Yn_Zn} satisfies the assumptions 
of the martingale central limit theorem, we first prove that the running maximum and minimum of $(\X_n)$ have sublinear growth.  

The argument we will use is similar to the one used to prove \cite[Lemma 3.2]{davis:annals}. The main idea is that $(\X_n)$ performs correlated random walk as long as it remains in previously visited territory, so if $\M_n - \m_n \geq L$ then the time to form a new extremum is stochastically at least the time for a correlated random walk to exit an interval of length $L$.

\begin{proposition}
\label{prop:max_min_gr}
 Let $(\X_n)$ be a p-rotor walk with $(\alpha, \beta)$-random initial configuration $\rho_0$ 
 as in \eqref{eq:random_cfg}. For every $p\in(0,1)$ and $\alpha,\beta\in[0,1]$, 
 \begin{equation*}
  \frac{\M_n}{n} \to 0\quad\text{and}\quad \frac{\m_n}{n} \to 0
 \end{equation*}
 in probability, as $n\to\infty$.
\end{proposition}
\begin{proof}
Fix $L>1$ and let
\begin{equation*}
 \tau_1=\inf\{n>0: \M_n-\m_n=L\}
\end{equation*}
be the first time when $(\X_n)$ has visited $L+1$ distinct points. For $k\geq 1$ consider the sequence
of stopping times
\begin{align*}
 \tau_{2k}&=\inf \big\{i>\tau_{2k-1}:\ \m_i < \X_i < \M_i\big\},\\
 \tau_{2k+1}&=\inf \big\{i>\tau_{2k}:\ \X_i<\m_{i-1} \text{ or }\X_i>\M_{i-1}\big\}.
\end{align*}
For each $k\geq 1$ the p-rotor walk reaches a previously unvisited vertex at time $\tau_{2k+1}$. It follows that
$\rho_{\tau_{2k+1}}(\X_{\tau_{2k+1}}) = \rho_{0}(\X_{\tau_{2k+1}})$ and $\X_{\tau_{2k+1}} \in \{\m_{\tau_{2k+1}}, \M_{\tau_{2k+1}}\}$. The conditional distribution
of $\tau_{2k+2} - \tau_{2k+1}$ given $\F_{\tau_{2k+1}}$ on the event $[\X_{\tau_{2k+1}} = \M_{\tau_{2k+1}}]$
is the geometric distribution with parameter $\alpha(1-p) + p(1-\alpha)$, since it represents the number of consecutive increases
of the maximum before changing direction. The process $(\X_n)$ stops increasing the maximum if the rotor at the current position points to the right in the 
initial configuration (with probability $\alpha$) and 
it is working (with probability  $1-p$) or if it points to the left (with probability $1-\alpha$) and it is broken
(with probability  $p$). Similarly the conditional distribution
of $\tau_{2k+2} - \tau_{2k+1}$ given $\F_{\tau_{2k+1}}$ on the event $[\X_{\tau_{2k+1}} = \m_{\tau_{2k+1}}]$
is the geometric distribution with parameter $\beta(1-p) + p(1-\beta)$. It follows that 
\begin{equation}
\label{eq:c}
\E[\tau_{2k+2}-\tau_{2k+1}] \leq C := \max\left\{\frac{1}{\alpha(1-p) + p(1-\alpha)}, \, \frac{1}{\beta(1-p) + p(1-\beta)} \right\}.
\end{equation}

For $k \geq 1$, in order to estimate the conditional distribution of $\tau_{2k+1}-\tau_{2k}$ given $\F_{\tau_{2k}}$, note that at time $\tau_{2k}$, the p-rotor walk is at distance $1$
from either its current maximum or the current minimum, and $\M_{\tau_{2k}}-\m_{\tau_{2k}} \geq L$.
Inside the already visited interval $I_k := \{\m_{\tau_{2k}},\ldots,\M_{\tau_{2k}}\}$
the rotors to the left of $\X_{\tau_{2k}}$ point right and the rotors to the right of $\X_{\tau_{2k}}$ point left.
These rotors coincide with the native environment  \eqref{eq:yn_rot_cfg} with the origin shifted to $\X_{\tau_{2k}}$. 
Therefore, starting at time $\tau_{2k}$ until the time $\tau_{2k+1}$ when it exits the interval $I_k$, the p-rotor walk is a correlated random walk with persistence $1-p$ (Definition~\ref{d.correlatedRW}). 
Thus, the conditional distribution of $\tau_{2k+1}-\tau_{2k}$ given 
$\F_{\tau_{2k}}$ is stochastically no smaller than the distribution of the time it takes a $(1-p)$-correlated
random walk started at $1$ to first visit the set $\{0,L\}$.
Denote by $E_L$ the expected hitting time of the set $\{0,L\}$ for a $(1-p)$-correlated random started at $1$, where the first step goes to $0$ with probability $p$ and to $2$ with probability $1-p$. From the law of large numbers 
\begin{equation*}
 \limsup_{n\to\infty}\dfrac{\sum_{k=1}^{n}(\tau_{2k}-\tau_{2k-1})}{\sum_{k=1}^{n-1}(\tau_{2k+1}-\tau_{2k})}\leq \frac{C}{E_L},
\end{equation*}
with $C$ given in \eqref{eq:c}.
On the other hand
\begin{equation*}
 \sum_{k=1}^n(\tau_{2k}-\tau_{2k-1})=\sum_{i=\tau_1}^{\tau_{2n}}\indicator\{\X_i<\m_{i-1} \text{ or } \X_i>\M_{i-1}\}
\end{equation*}
and 
\begin{equation*}
\sum_{k=1}^{n-1}(\tau_{2k+1}-\tau_{2k})\leq \tau_{2n}.
\end{equation*}
Then we have
\begin{align*}
\limsup_{n\to\infty}\frac{1}{n}(\M_n - \m_n) &= \limsup_{n\to\infty}\frac{1}{n}\left(L + \sum_{i=\tau_1}^{n} \indicator\{\X_i<\m_{i-1} \text{ or } \X_i>\M_{i-1}\}\right)\\
 &\leq \limsup_{n\to\infty}\frac{1}{\tau_{2n}}\left(\sum_{i=\tau_1}^{\tau_{2n}} \indicator\{\X_i<\m_{i-1} \text{ or } \X_i>\M_{i-1}\}\right)\\
 & \leq \limsup_{n\to\infty}\frac{\sum_{k=1}^n(\tau_{2k}-\tau_{2k-1})}{\sum_{k=1}^{n-1}(\tau_{2k+1}-\tau_{2k})}\leq \frac{C}{E_L}.
\end{align*}
Proposition \ref{prop:undisturbed_clt} (which also implies the recurrence of the correlated random
walk) together with the Portmanteau theorem 
yields that $\sup_{L> 1}E_L=\infty$, which gives
\begin{equation*}
\limsup_{n\to\infty}\frac{1}{n}(\M_n - \m_n) = 0.
\end{equation*}
Since $\M_n \leq \M_n - \m_n$ and $\abs{\m_n} \leq \M_n - \m_n$, the proposition follows.
\end{proof}

Now we obtain the scaling limit of the martingale portion $(\Y_n)$ of the $p$-rotor walk. Note that the constant factor in front of the Brownian motion here is different from the $\sqrt{\frac{1-p}{p}}$ we are ultimately aiming for in the scaling limit of $(\X_n)$.

\begin{theorem}
Let $(\Y_n)$ be the martingale defined in \eqref{eq:Yn_Zn}. Then on the space $\mathcal{C}[0,1]$
\label{thm:yn_convergence}
\begin{equation*}
\left\{\frac{\Y(nt)}{\sqrt{n}},\, t\in [0,1] \right\}
\weaklyto \left\{2\sqrt{p(1-p)}\B(t),\,t\in [0,1]\right\}
\quad\text{as } n\to\infty.
\end{equation*}
\end{theorem}
\begin{proof}
We check the conditions of the martingale central limit theorem from Theorem \ref{thm:martingale_clt}. As in the proof of Proposition \ref{prop:undisturbed_clt} the first
condition of Theorem \ref{thm:martingale_clt} is satisfied since $\xi_k$ is bounded. 
Similarly to \eqref{eq:vn_variation} the following equality holds
\begin{equation*}
\mathsf{V}_n = n-\sum_{k=1}^{n}\mathbb{E}[\Delta_k|\F_{k}]^2.
\end{equation*}
We use once again that $\Delta_k=\X_{k+1}-\X_{k}=B_{k}\rho_{k}(\X_{k})$, where $B_{k}$ is independent of $\F_{k}$ with $\mathbb{E}[B_{k}]=(2p-1)$.
On the event $[\m_{k-1}\leq \X_k \leq \M_{k-1}]$ the rotor $\rho_k(\X_k)$ is $\F_k$-measurable, since it points into the direction of the last
exit from $\X_k$. On the other hand, on the event $[\m_{k-1} > \X_k \text{ or }  \M_{k-1}<\X_k]$ the rotor $\rho_k(\X_k) = \rho_0(\X_k)$
is still in its initial state, which is independent of $\F_k$. Thus
\begin{equation}\label{eq:delta_square}
\begin{aligned}
\mathbb{E}[\Delta_k|\F_{k}]^2 & = (2p-1)^2 (1-2\beta)^2\indicator\{\X_{k}<\m_{k-1}\}\\
& + (2p-1)^2\rho_k(\X_k)^2\indicator\{\m_{k-1}\leq \X_k \leq \M_{k-1}\}\\
& + (2p-1)^2(2\alpha-1)^2\indicator\{\X_k>\M_{k-1}\}.
\end{aligned}
\end{equation}
Moreover, because
$
\m_n = -\sum_{k=1}^n \indicator\{\X_{k}<\m_{k-1}\}$,
$\M_n = \sum_{k=1}^n \indicator\{\X_{k}>\M_{k-1}\}$
and 
\begin{equation*}
 \sum_{k=1}^n\indicator\{\m_{k-1}\leq \X_k \leq \M_{k-1}\}=n-\M_n+\m_n,
\end{equation*}
equations \eqref{eq:vn_variation} and \eqref{eq:delta_square} imply that
\begin{equation*}
 \frac{\mathsf{V}_n}{n}=1-(2p-1)^2\bigg\{-(1-2\beta)^2\frac{\m_n}{n}
 +\frac{n}{n}-\frac{\M_n}{n}+\frac{\m_n}{n}
 +(2\alpha-1)^2\frac{\M_n}{n}\bigg\}.
\end{equation*}
This together with Proposition \ref{prop:max_min_gr} finally yields
\begin{equation*}
\frac{\mathsf{V}_n}{n}\to 1-(2p-1)^2=4p(1-p)>0. 
\end{equation*}
Since all conditions from Theorem \ref{thm:martingale_clt} are satisfied, with $\sigma=2\sqrt{p(1-p)}$,
the claim follows.
\end{proof}

Recall that for all $n\geq 1$,
the p-rotor walk satisfies an equation of the form
\begin{equation}
\label{eq:xn_equation}
\X_n = \f_n + a \M_n + b \m_n,
\end{equation}
with $a,b < 1$ and $\f_n$ given in Proposition \ref{prop:xn_decomposition}. 
\begin{lemma}\label{lem:conv_fn}
Let $\f_n$ as defined in Proposition \ref{prop:xn_decomposition}. Then $\f_n$,
when rescaled by $\sqrt{n}$ converges weakly on $\mathcal{C}[0,1]$ to a Brownian motion:
\begin{equation*}
 \left\{\frac{\f(nt)}{\sqrt{n}},\, t\in[0,1] \right\}
\weaklyto \left\{\sqrt{\frac{1-p}{p}}\B(t),\,t\in [0,1]\right\}
\quad\text{as } n\to\infty.
\end{equation*}
\end{lemma}
\begin{proof}
 Because $\f_n$ is $\Y_n$ plus a bounded quantity, rescaled by $1/2p$, this implies
 that $\f_n$ has the same scaling limit as $\frac{1}{2p}\Y_n$. This together with 
 Theorem \ref{thm:yn_convergence} gives the claim.
\end{proof}

In the remainder of this section we will show that this
implies the weak convergences of $\X(nt)/\sqrt{n}$ to a doubly perturbed Brownian motion.
We shall use the following identities from \cite[page 243]{carmona_petit_yor_1998} characterizing the maximum and minimum of a solution to an
equation of the form \eqref{eq:xn_equation}. We include the proof for completeness.
\begin{lemma}
\label{lem:max_equation}
Let $\M_n$ and $\m_n$ be the running maximum and minimum of a process $(\X_n)$ satisfying \eqref{eq:xn_equation}.
Then
\begin{align*}
\M_n = \frac{1}{1-a} \max_{k\leq n} \left(\f_k + \frac{b}{1-b} g_k\right)
\quad\text{and}\quad
\m_n = \frac{1}{1-b} \min_{k\leq n} \left(\f_k + \frac{a}{1-a} G_k\right),
\end{align*}
where $g_k = \min_{l\leq k}\big(\f_l + a\M_l\big)$ and $G_k =  \max_{l\leq k} 
\big(\f_l + b \m_l\big)$,
\end{lemma}
\begin{proof}
From \eqref{eq:xn_decomposition} we have $\X_n - a \M_n = \f_n + b \m_n$. Taking the maximum 
over $n$ on both sides gives
$$(1-a) \M_n = \max_{k\leq n} \big(\f_k + b\m_n\big).$$
Similarly 
$$(1-b) \m_n = \min_{k\leq n}\big(\f_k + a \M_n\big).$$
Solving for the running maximum $\M_n$ and for the running minimum $\m_n$
gives the claim.
\end{proof}

We shall also use the following easy inequality.
\begin{proposition}
\label{prop:max_diff}
Let $(x_k)_{k\geq 0}$ be a sequence of real numbers. Then for all $n,j\in\mathbb{N}_0$,
\begin{equation}
\label{eq:max_diff}
\max_{k\leq n+j} x_k - \max_{k\leq j} x_k \leq \max_{k\leq n} \big(x_{j+k} - x_j\big).
\end{equation}
\end{proposition}
\begin{proof}
If the left hand side of \eqref{eq:max_diff} is equal to zero, the statement is trivially true.
Now assume that $\max_{k\leq n+j} x_k > \max_{k\leq j} x_k$. It follows that
\begin{equation*}
\max_{k\leq n+j} x_k = \max_{j \leq k \leq n+j} x_k = \max_{k \leq n} x_{k+j}.
\end{equation*}
Hence 
$$\max_{k\leq n+j} x_k - \max_{k\leq j} x_k \leq \max_{k \leq n} x_{k+j} - x_j =
\max_{k \leq n} \big(x_{k+j} - x_j\big).$$
\end{proof}

\begin{lemma}
\label{lem:max_diff_bound}
There exists a constant $C > 0$ such that
\begin{equation*}
\abs{\M_{j+n} - \M_j} \leq C \max_{k\leq n} \abs{\f_{j+k}-\f_j} \quad \text{and}\quad \abs{\m_{j+n} - \m_j} \leq C \max_{k\leq n} \abs{\f_{j+k}-\f_j},
\end{equation*}
for all $j,n\geq 0$.
\end{lemma}
\begin{proof}
By Lemma \ref{lem:max_equation} and Proposition \ref{prop:max_diff}
\begin{align}
\label{eq:max_diff_bound}
\begin{aligned}
\M_{j+n} - \M_j &= \frac{1}{1-a}\left\{ \max_{k\leq j+n}\left(\f_k + \frac{b}{1-b}g_k\right) - 
                                                                \max_{k\leq j}\left(\f_k + \frac{b}{1-b}g_k\right)
                                                        \right\} \\
                      &\leq \frac{1}{1-a} \max_{k\leq n} \left\{(\f_{j+k} - \f_j) + \frac{b}{1-b}(g_j - g_{j+k})\right\}.
\end{aligned}
\end{align}
We shall distinguish two cases. Let first $b \leq 0$. Since 
$$g_j - g_{j+k} = \min_{l\leq j}\big(\f_l + a\M_l\big) - \min_{l\leq j+k}\big(\f_l + a\M_l\big) \geq 0,$$
we get the bound
\begin{equation*}
\M_{j+n} - \M_j \leq \frac{1}{1-a} \max_{k\leq n} \big(\f_{j+k} - \f_j\big).
\end{equation*}
If $b > 0$ we can apply again Proposition \ref{prop:max_diff} to obtain
\begin{align*}
g_j - g_{j+k} &= \min_{l\leq j}\big(\f_l + a\M_l\big) - \min_{l\leq j+k}\big(\f_l + a\M_l\big) \\
              &= \max_{l\leq j+k}\big(-\f_l - a\M_l\big) - \max_{l\leq j}\big(-\f_l - a\M_l\big) \\
              &\leq \max_{l\leq k} \big((\f_j-\f_{j+l}) + a(\M_j - \M_{j+l})\big) \\
              &\leq \max_{l\leq k} \big(\f_j-\f_{j+l}\big),
\end{align*}
where the last inequality follows from the fact $\M_j - \M_{j+l}\leq 0$ and that $b >0$ implies that also $a>0$.  
Together with \eqref{eq:max_diff_bound} this gives 
\begin{align*}
\M_{j+n} - \M_j &\leq \frac{1}{1-a} \max_{k\leq n} \left\{(\f_{j+k} - \f_j) + \frac{b}{1-b}\max_{l\leq k} \big(\f_j-\f_{j+l}\big)\right\} \\
                &\leq \frac{1}{1-a} \max_{k\leq n} \left\{\abs{\f_{j+k} - \f_j} + \frac{b}{1-b}\max_{l\leq k} \abs{\f_{j+l}-\f_j}\right\} \\
                &\leq \frac{1}{1-a} \left\{ \max_{k\leq n} \abs{\f_{j+k} - \f_j} + \frac{b}{1-b} \max_{k\leq n} \abs{\f_{j+k} - \f_j} \right\}\\
                &= \frac{1}{(1-a)(1-b)} \max_{k\leq n} \abs{\f_{j+k} - \f_j}.
\end{align*}
The upper bound for the differences of the minimum follows from the same argument with the roles of $a$ and $b$ exchanged. By setting
$$C = \max\left\{\frac{1}{1-a}, \frac{1}{1-b}, \frac{1}{(1-a)(1-b)}\right\}$$ the claim follows.
\end{proof}

\begin{proposition}\label{prop:Mm_tight}
For $n\geq 1$ let $M_n(t) = \frac{\M(nt)}{\sqrt{n}}$ and $m_n(t) = \frac{\m(nt)}{\sqrt{n}}$ be the processes obtained by linearly interpolating and rescaling the running maximum and the running minimum of $(\X_n)$, respectively.
Then $(M_n)_{n\geq 1}$ and $(m_n)_{n\geq 1}$ are tight sequences in $\mathcal{C}[0,1]$.
\end{proposition}
\begin{proof}
We show the tightness only for the rescaled maximum $M_n$. By symmetry, the same argument also applies to $m_n(t)$. Since $M_n(0) = 0$ for all $n\geq 1$ by Theorem
7.3 of \cite{billingsley_2nd_edition} we only need to show that for all $\epsilon > 0$
\begin{equation*}
\lim_{\delta\to 0}\limsup_{n\to\infty} \Pb\left[\sup_{\abs{s-t} < \delta} \abs{M_n(s) - M_n(t)} \geq \epsilon\right] = 0.
\end{equation*}
Undoing the rescaling, this is equivalent to showing that
\begin{equation*}
\lim_{\delta\to 0}\limsup_{n\to\infty} \Pb\left[\max_{k < n\delta} \abs{\M_{j+k} - \M_j} \geq \epsilon\sqrt{n}\right] = 0,
\end{equation*}
for all $j\geq 0$ and all $\epsilon > 0$.

By Lemma \ref{lem:max_diff_bound} we have $\max_{k < n\delta} \abs{\M_{j+k} - \M_j} \leq C \max_{k\leq n\delta} \abs{\f_{j+k} - \f_j}$.
Thus
\begin{align*}
\lim_{\delta\to 0}\limsup_{n\to\infty} \Pb\left[\max_{k < n\delta} \abs{\M_{j+k} - \M_j} \geq \epsilon\sqrt{n}\right]\leq
\lim_{\delta\to 0}\limsup_{n\to\infty} \Pb\left[\max_{k < n\delta} \abs{\f_{j+k} - \f_j} \geq \frac{\epsilon}{C}\sqrt{n}\right].
\end{align*}
By Lemma~\ref{lem:conv_fn} the right side is zero: $\frac{\f(nt)}{\sqrt{n}}$ converges weakly to a Brownian motion and is therefore a tight sequence. Hence $(M_n)$ is tight.
\end{proof}

Now we turn to the proof of the main result. 

\begin{proof}[Proof of Theorem \ref{thm:scaling_lim}]
To recall the setup, $X_n(t) = \X(nt) /\sqrt{n}$ is the linearly interpolated rescaling of the p-rotor walk
with the initial configuration \eqref{eq:random_cfg}.
Using the decomposition \eqref{eq:xn_decomposition} we have
\begin{equation}
\label{eq:rescaled_decomposition}
X_n(t) = W_n(t) + a M_n(t) + b m_n(t).
\end{equation}
where $M_n(t)= \M(nt)/ \sqrt{n}$ and $m_n(t)= \m(nt) / \sqrt{n}$ are the linearly interpolated rescalings of the maximum and minimum of
$(\X_n)$, respectively; and $W_n(t)= \f(nt) / \sqrt{n}$
with $\f_n$ defined in Proposition \ref{prop:xn_decomposition}.

The sequence $W_n$ converges weakly in $\mathcal{C}[0,1]$ to a Brownian motion by Lemma \ref{lem:conv_fn}, and the sequences $M_n$ and $m_n$ are tight by Proposition \ref{prop:Mm_tight}. So $X_n$ is a sum of three tight sequences, hence tight. By Prohorov's theorem, every subsequence of $X_n$ contains a further subsequence
that converges weakly in $\mathcal{C}[0,1]$. Let $X_{n_j}$ be a convergent subsequence with 
a weak limit point which we denote by $X$.  

Now we apply the continuous mapping theorem, using the map $\Theta : \mathcal{C}[0,1] \to \mathcal{C}[0,1]$ given by
	\[ \Theta(h)(t) = h(t) - a \sup_{0 \leq s \leq t} h(s) - b \inf_{0 \leq s \leq t} h(s). \]
Rearranging the terms in \eqref{eq:rescaled_decomposition}, we have 
	\[ \Theta(X_{n_j}) = W_{n_j}. \]
By the continuous mapping theorem, the left side converges weakly in $\mathcal{C}[0,1]$ to $\Theta(X)$, and by Lemma~\ref{lem:conv_fn} the right side converges weakly to a Brownian motion $\sqrt{\frac{1-p}{p}} \B$.  We conclude that
	\[ X - a X^{\mathsf{sup}} - b X^{\mathsf{inf}} \; \dequal \; \sqrt{\frac{1-p}{p}} \B \]
as processes on $[0,1]$. It follows that $\sqrt{\frac{p}{1-p}} X$ is a solution of the implicit equation \eqref{eq:perturbed_brownian_motion}, i.e.~it is an $(a,b)$-perturbed 
Brownian motion.

Since $a,b\in(-\infty,1)$, the equation \eqref{eq:perturbed_brownian_motion} uniquely determines the law of $X$ (see Chaumont and Doney \cite[Theorem 2]{chaumont_doney_1999}, who show something stronger:  \eqref{eq:perturbed_brownian_motion} has a \emph{pathwise} unique solution, which is almost surely continuous and adapted to the filtration of $\B$). Hence every convergent subsequence of $X_n$ has the same weak limit point in $\mathcal{C}[0,1]$, which implies that the sequence $X_n$ itself converges weakly (see the Corollary to
Theorem 5.1 in \cite{billingsley_2nd_edition}). 

The same argument proves weak convergence of $(X_n(t))_{0 \leq t \leq T}$ to an $(a,b)$-perturbed Brownian motion in $\mathcal{C}[0,T]$, for each fixed $0<T<\infty$.
Since weak convergence in $\mathcal{C}[0,\infty)$ is defined as weak convergence in  $\mathcal{C}[0,T]$ for all $0<T<\infty$, the proof is complete.
\end{proof}

\begin{remark}\label{rem:in_cfg}
If we change any finite number of rotors in the initial rotor configuration \eqref{eq:random_cfg}, the scaling limit
of the p-rotor walk will be the same as in Theorem \ref{thm:scaling_lim}.
The only difference in the proof is that in the compensator $\mathsf{Z}_n$ from Proposition \ref{prop:z_n}
there will be some additional terms of order $1$, which, after rescaling by $\sqrt{n}$ and letting $n$ go to infinity, go to zero.
\end{remark}

If one of the parameters $\alpha$ or $\beta$ are $0$ in the initial configuration
$\rho_0$ given in \eqref{eq:random_cfg}, then the scaling limit can be determined explicitly and it
is a one-sided perturbed Brownian motion, as in Corollary \ref{cor:one-sided}.
The proof is a simple calculation, which we state here for completeness.
\begin{proof}[Proof of Corollary \ref{cor:one-sided}.]
By letting $\beta=0$, we have $b=0$ and the scaling limit of $(\X_n)$
satisfies the implicit equation
\begin{equation}
 X(t)=\sqrt{\frac{1-p}{p}} \B(t) + a X^{\sup}(t),
\end{equation}
which implies that
\begin{equation*}
(1-a)X^{\sup}(t)= \sup_{s\leq t}\big(X(s)-a X^{\sup}(t)\big)=\sqrt{\frac{1-p}{p}}\sup_{s\leq t}\B(s)=\sqrt{\frac{1-p}{p}}\mathcal{M}(t).
\end{equation*}
Thus
\begin{equation*}
 X^{\sup}(t)=\sqrt{\frac{1-p}{p}}\cdot \frac{1}{1-a}\mathcal{M}(t)
\end{equation*}
and
\begin{equation*}
  X(t)=\sqrt{\frac{1-p}{p}}\Big( \B(t) + \frac{a}{1-a} \mathcal{M}(t)\Big).
\end{equation*}
On the other hand, since $a=\frac{\alpha(2p-1)}{p}$, we have
\begin{equation*}
\frac{a}{1-a}=\lambda_{p,\alpha}= \frac{\alpha(2p-1)}{\alpha(1-p)+p(1-\alpha)}
\end{equation*}
and this proves the claim.
\end{proof}
By symmetry, the same scaling limit can be obtained if we take $\alpha=0$, and the p-rotor walk
with $(0,\beta)$- rescaled by a factor of $\sqrt{n}$ converges weakly to a Brownian motion perturbed at its minimum.

We conclude with the proof of recurrence of p-rotor walk.
\begin{proof}[Proof of Proposition \ref{prop:recurrence}]
In \cite[Theorem 1.1]{chaumont_doney_2000} a version of the law of iterated logarithm
was proved for doubly perturbed Brownian motion $\mathcal{X}_{a,b}(t)$. In particular this implies
that $\limsup_{t\to\infty} \mathcal{X}_{a,b}(t) = +\infty$ and $\liminf_{t\to\infty} \mathcal{X}_{a,b}(t) = -\infty$. This further implies that $\mathcal{X}_{a,b}(t)$ crosses
$0$ an infinite number of times. Hence $\Pb[\X_n = 0 \text{ infinitely often}] = 1$ by
Theorem \ref{thm:scaling_lim} and the fact that p-rotor walk is a nearest neighbor process.
\end{proof}

\section{Higher dimensions and longer jumps}\label{sec:questions}

\paragraph{Two dimensions.}
Consider the following nearest-neighbor walk $(X_n)$ in $\Z^2$, with rotors $\rho_n: \Z^2 \to \{(0,1),(0,-1),(1,0),(-1,0)\}$. Initially the rotors $\rho_0$ are i.i.d.\ uniform. At each time step, the rotor $\rho_n(X_n)$ at the walker's current location either resets to a uniformly random direction (with probability $p$) or rotates clockwise $90^\circ$ (with probability $1-p$), and then the walker follows the new rotor: $X_{n+1} = X_n + \rho_{n+1}(X_n)$. This walk interpolates between the uniform rotor walk ($p=0$) and simple random walk ($p=1$).

\textbf{Question 1}: Prove that $(X_n)$ is recurrent for all $p \in (0,1)$.

\textbf{Question 2}: What is the scaling limit of $(X_n)$, for $p \in (0,1)$?

Following the outline of this paper, one could first seek to understand the process $(X_n)$ in its native environment, which is when $\rho_0$ is a variant of the uniform spanning tree (oriented toward the origin).

\paragraph{Locally Markov walks.}
The p-rotor walk  $(\X_n)$ on $\Z$ is not a Markov chain, 
but for each vertex $x\in\Z$, the sequence of successive exits
from $x$ is a Markov chain. 
More formally, we could make the following definition. 

\begin{definition}
A walk $(\X_n)_{n\geq 0}$ is \emph{locally Markov} if for each vertex $x$ the sequence
$ (X_{t^x_k + 1} )_{k \geq 1}$ is a Markov chain, where $t^x_1 < t^x_2 <
\cdots$ are the times of successive visits to $x$. 
\end{definition}
In the case of the p-rotor walk on $\Z$, for each $x\in\Z$, the Markov chain  $ (\X_{t^x_k + 1} )_{k \geq 1}$
has state space $\{x-1,x+1\}$.  A natural generalization allows longer range jumps while preserving the local Markov property.
Specifically, suppose that each vertex $x$ has a finite state space $Q_x\subset \Z$, and let $(\mathsf{p}_x(q,q'))_{q,q'\in Q_x}$
be a transition matrix on $Q_x$.  The total state of the system is specified by the current location the walker in $\Z$, along with a local state $q_x \in Q_x$ for each $x \in \Z$.
A walker located at $x$ and finding local state $q=q_x$, first changes the local state to $q'$ with 
probability $\mathsf{p}_x(q,q')$ and then moves to $x+q'$. Let us denote by $(X_n)$ the location of the walker after $n$ such steps.
For each $x\in\Z$, the sequence $ (X_{t^x_k + 1} )_{k \geq 1}$ of successive exits from $x$ is a Markov chain on the finite state space $\{x+q:q\in Q_x\}$.  Assume that $\mathsf{p}_x$ is irreducible, so that this chain has a unique stationary distribution $\pi_x$.

The process $(X_n)$  has a corresponding coarse-grained Markov chain $(Y_n)$ on $\Z$ whose transition probabilities are  $p(x,y) = \pi_x(y)$.
Comparing these two processes raises a number of questions.

\textbf{Question 3}: 
Is it possible that $(Y_n)$ is recurrent but $(X_n)$ is transient? If the local chain $ (X_{t^x_k + 1} )_{k \geq 1}$ is only assumed to be \emph{hidden} Markov, then the answer is yes, as shown by Pinsky and Travers \cite{pinski-travers-15}.
For excited random walks with non-nearest neighbor steps, the question of transience/recurrence was investigated in \cite{davis-peterson-15}. It may be useful to understand if one can adapt their results in our setting of locally Markov walks.

\textbf{Question 4}: Suppose $(Y_n)$ has no drift (i.e., each $\pi_x$ has mean $x$). Under what conditions is the scaling limit of $(X_n)$ a perturbed Brownian motion like in Theorem \ref{thm:scaling_lim}? 

\textbf{Question 5}:
What is the scaling limit of $(X_n)$ in the case when $(Y_n)$ has drift?

\section*{Acknowledgements}

We thank B\'{a}lint Toth for bringing reference \cite{toth:rayknight} to our attention, Nick Travers for a helpful conversation, and Boyao Li for helpful comments on an early draft. We also thank to Tal Orenshtein, Jonathon Peterson, and Elena Kosygina for pointing us the connection between excited random walks with Markovian cookie stacks and our model.

\bibliography{p-walk}{}
\bibliographystyle{alpha_arxiv}

\textsc{Wilfried Huss, Institute of Discrete Mathematics, Graz University of Technology, 8010 Graz, Austria.}
\texttt{huss@math.tugraz.at} \\ \url{http://www.math.tugraz.at/~huss}

\textsc{Lionel Levine, Department of Mathematics, Cornell University, Ithaca, NY 14850.}
\url{http://www.math.cornell.edu/~levine}

\textsc{Ecaterina Sava-Huss, Institute of Discrete Mathematics, Graz University of Technology, 8010 Graz, Austria.}
\texttt{sava-huss@tugraz.at}\\
\url{http://www.math.tugraz.at/~sava} 

\end{document}

%% file: sim_trajectory_arxiv.tex
\begin{center}
\includegraphics[width=0.49\textwidth]{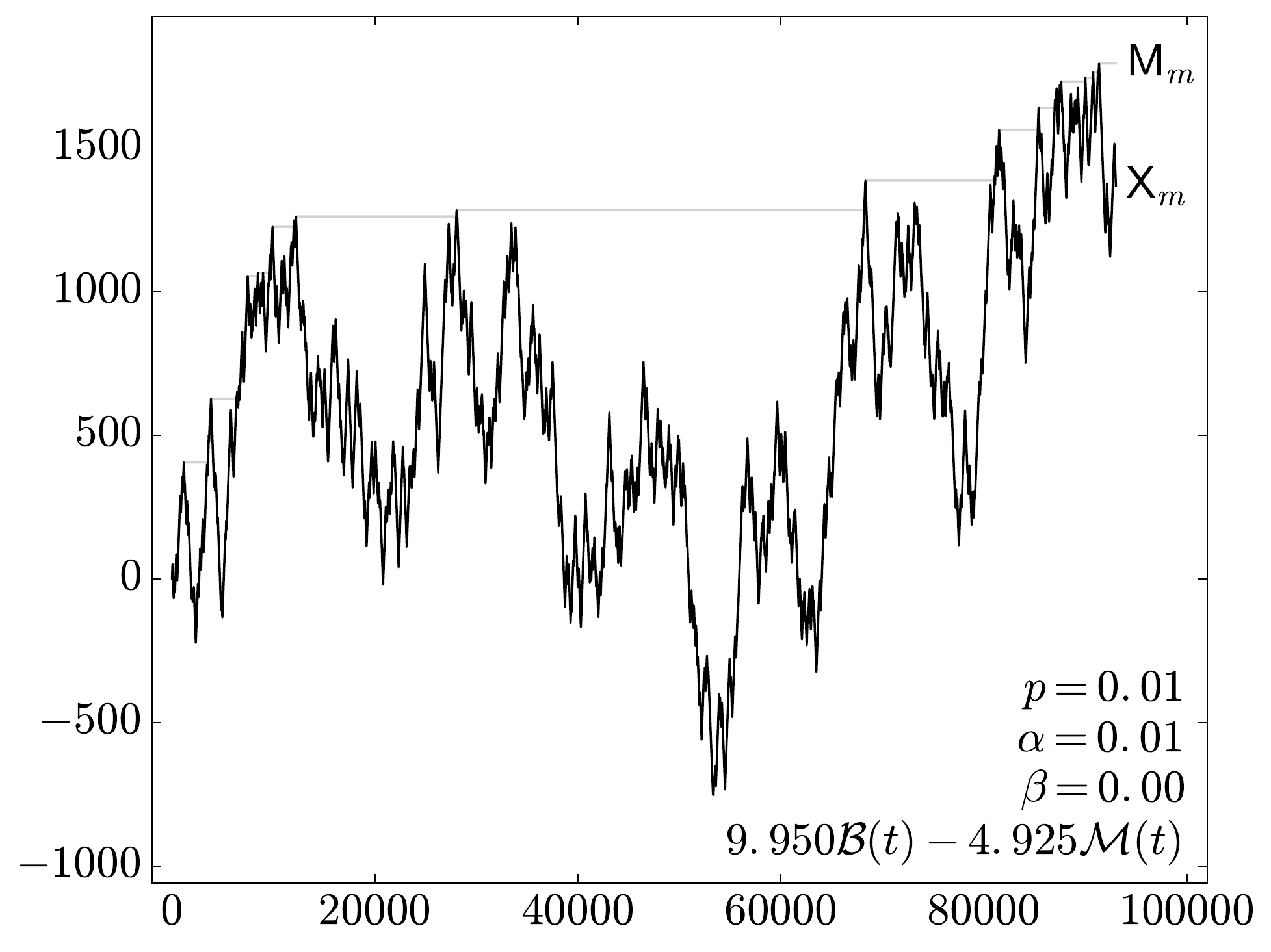}
\hfill
\begin{tikzpicture}[inner sep = 0]
\node (g) at (0,0) {\includegraphics[width=0.49\textwidth]{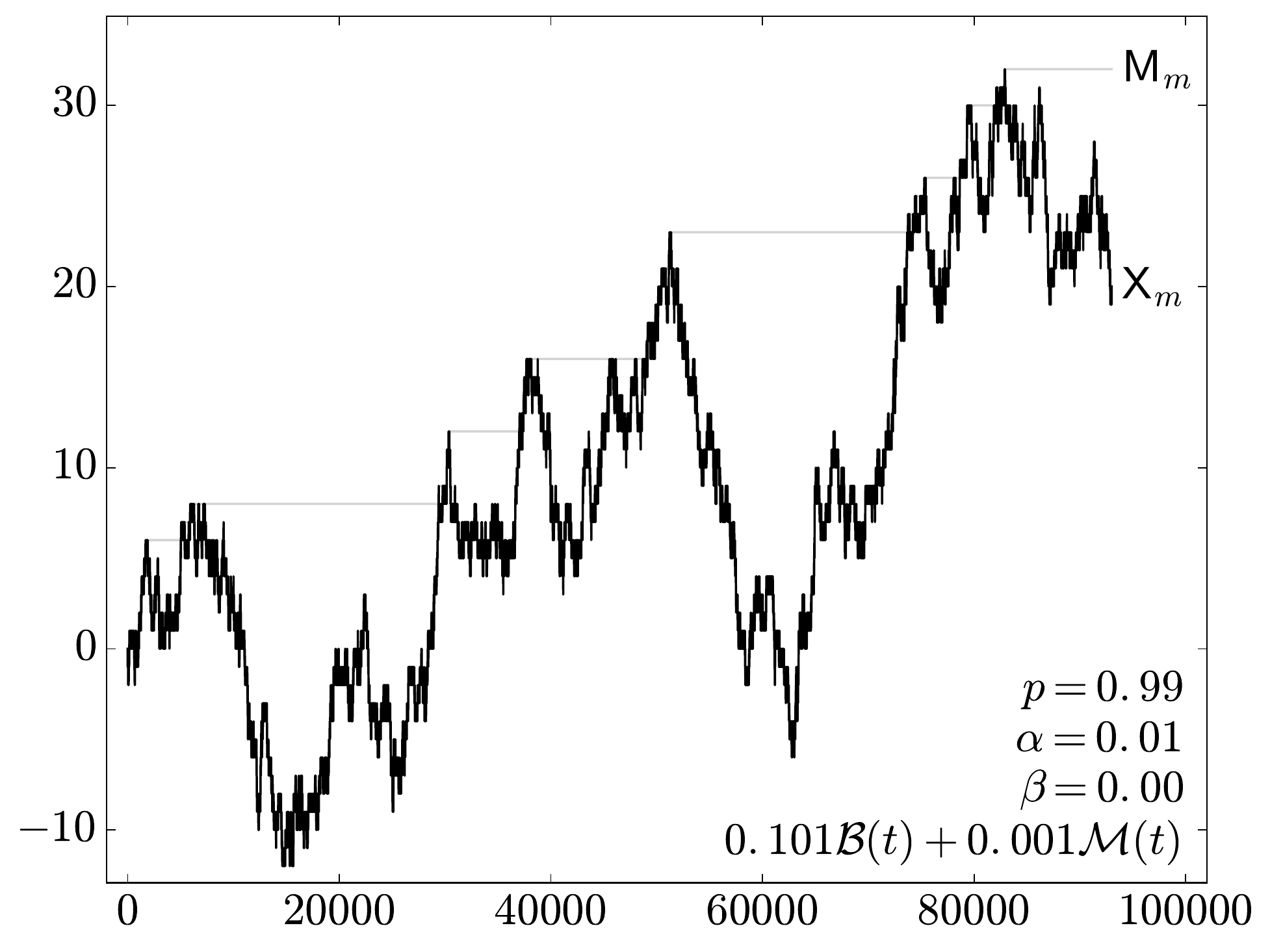}};
\node (inset) at (-1.9,1.7) {\includegraphics[width=0.15\textwidth]{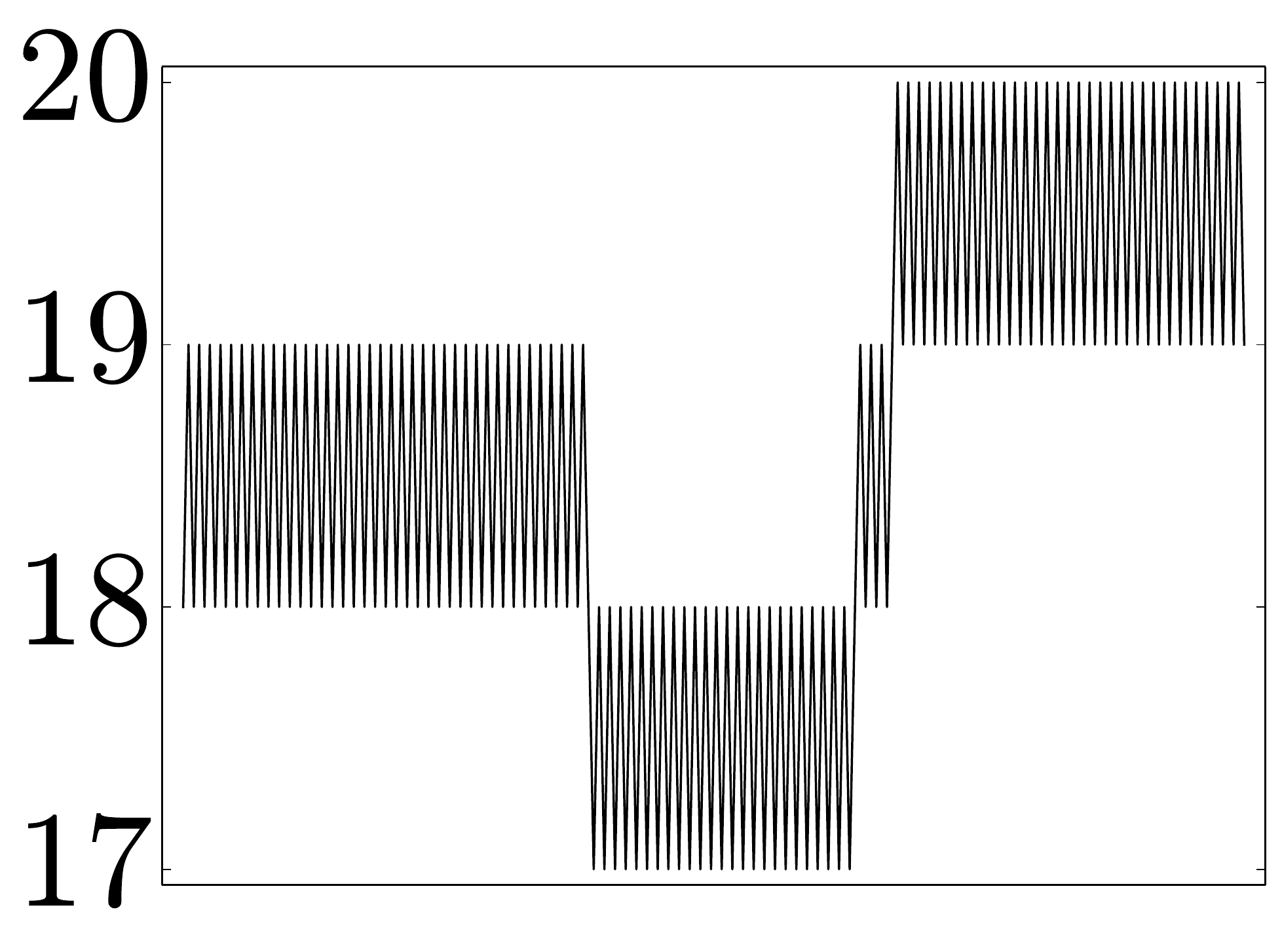}};
\node[rectangle, draw, minimum height=0.4cm, minimum width=0.02cm, line width = 0cm] (rect) at (0.14,0.92) {};
\draw[line width = 0cm] (rect.south west) -- ($(inset.south west) + (0.3,0.125)$);
\draw[line width = 0cm] (rect.north east) -- ($(inset.north east) + (-0.03, -0.13)$);
\end{tikzpicture}
\\[2ex]
\includegraphics[width=0.49\textwidth]{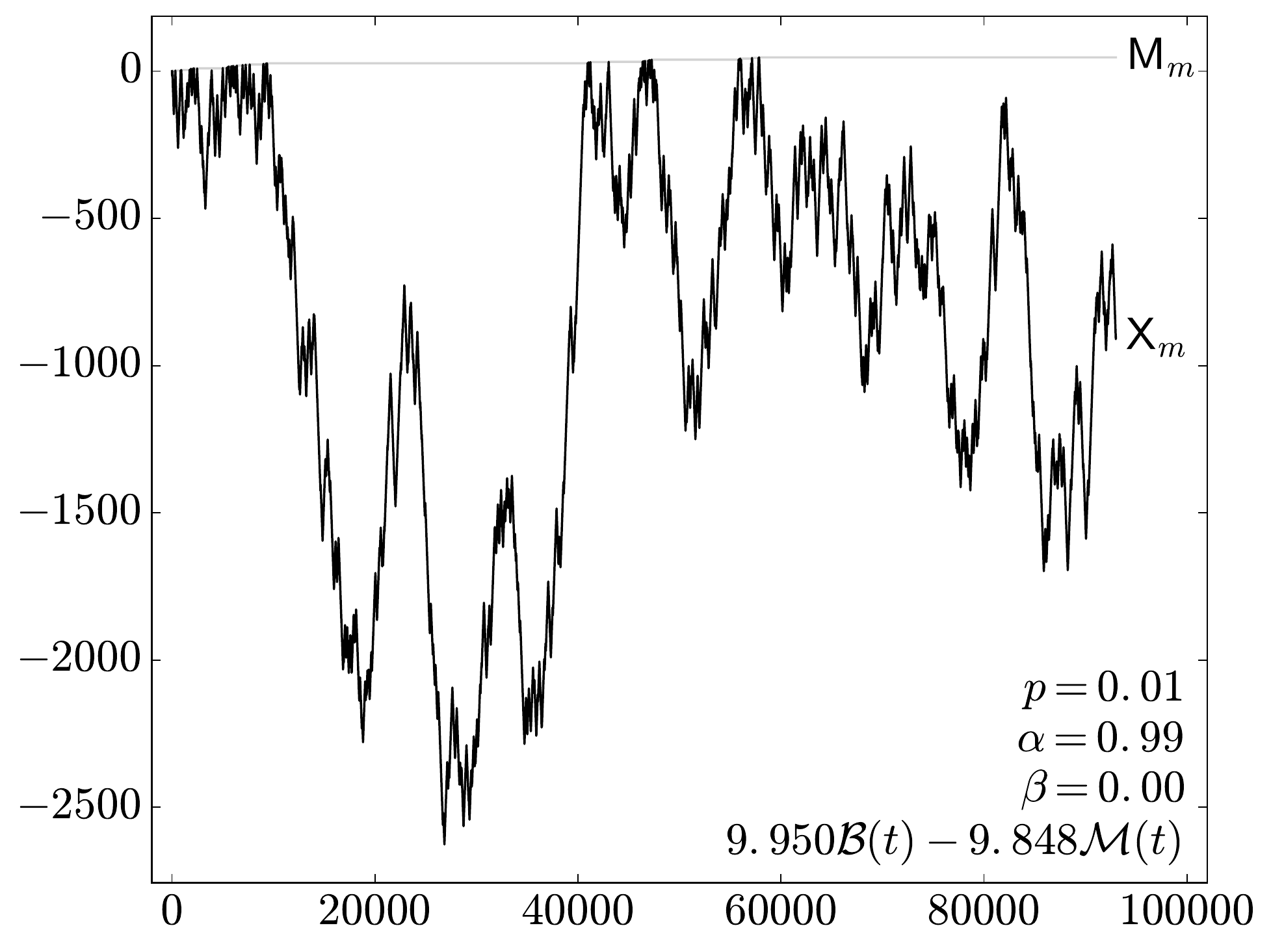}
\hfill
\includegraphics[width=0.49\textwidth]{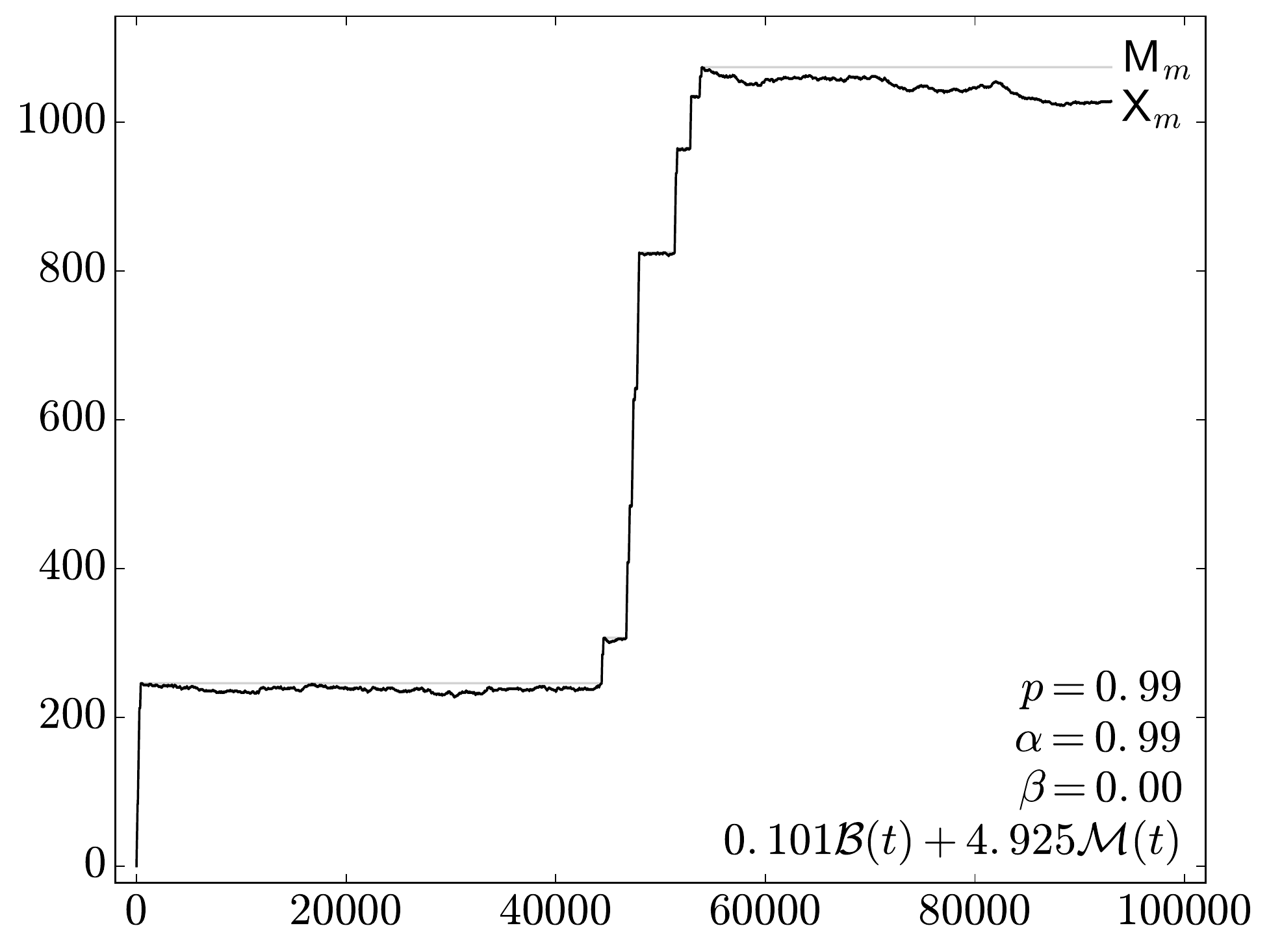}
\end{center}

%% file: figure_environment.tex
\begin{tikzpicture}[scale=0.45]
\draw[lightgray] (-12.5,0) -- (12.5,0);

\foreach \x in {-12,...,12} {
	\fill[black] (\x,0) circle (1.6pt);
}

\foreach \x/\y in {-12/-1, -11/1, -10/-1, -9/-1, -8/1, -7/1, -6/-1, -5/1, -4/1, -3/1, -2/1, -1/1, 0/-1, 1/-1, 2/-1, 3/-1, 4/-1, 5/-1, 6/1, 7/-1, 8/1, 9/1, 10/-1, 11/-1, 12/1} {
    \draw[black, -stealth, shorten >=0.2pt, line width = 1] (\x,0) -- (\x + 0.5\y, 0);
}

\draw (4,0.4) -- (4,-0.4);
\draw (-5,0.4) -- (-5,-0.4);

\fill[black] (0,0) circle (4pt);

\coordinate[label=90:{$\M_{n}$}] (Mn) at (4,0.4);
\coordinate[label=90:{$\X_{n}$}] (Xn) at (0,0.4);
\coordinate[label=90:{$\m_{n}$}] (mn) at (-5,0.4);
\end{tikzpicture}